\documentclass{article}%
\usepackage{amssymb,latexsym}
\usepackage{amsfonts}
\usepackage{amsmath}
\usepackage[colorlinks=true, pdfstartview=FitV, linkcolor=blue,
citecolor=blue, urlcolor=blue]{hyperref}
\usepackage{color}
\usepackage{amssymb}
\usepackage{graphicx}%
\providecommand{\U}[1]{\protect\rule{.1in}{.1in}}
\newtheorem{theorem}{Theorem}

\newtheorem{corollary}[theorem]{Corollary}

\newtheorem{proposition}[theorem]{Proposition}
\newtheorem{remark}[theorem]{Remark}

\newenvironment{proof}[1][Proof]{\noindent\textbf{#1.} }{\ \rule{0.5em}{0.5em}}
\begin{document}

\title{On reciprocity formula of character Dedekind sums and the integral of
products of Bernoulli polynomials}
\author{M. Cihat Da\u{g}l\i\ and M\"{u}m\"{u}n Can\thanks{Corresponding author}\\Department of Mathematics, Akdeniz University, 07058-Antalya, Turkey\\ E--mail: mcihatdagli@akdeniz.edu.tr, mcan@akdeniz.edu.tr}
\date{}
\maketitle
\begin{abstract}
We give a simple proof for the reciprocity formulas of character Dedekind sums associated with two primitive characters, whose modulus need not to be same, by utilizing the character analogue of the Euler--MacLaurin summation formula. Moreover, we extend known results on the integral of products of Bernoulli polynomials by considering the integral
\[
\int\limits_{0}^{x}B_{n_{1}}\left(  b_{1}z+y_{1}\right)  \cdots B_{n_{r}}\left(  b_{r}z+y_{r}\right)  dz,
\]
where $b_{l}$ $(b_{l}\neq 0)$ and $y_{l}$ $(1\leq l\leq r)$ are real numbers. As a consequence of this integral we establish a connection between the reciprocity relations of sums of products of Bernoulli polynomials and of the Dedekind sums.



\end{abstract}

\section{Introduction}

Let%
\[
\left(  \left(  x\right)  \right)  =%
\begin{cases}
x-\left[  x\right]  -1/2, & \text{if }x\in\mathbb{R}\backslash\mathbb{Z}%
\text{,}\\
0, & \text{if }x\in\mathbb{Z}\text{,}%
\end{cases}
\]
with $[x]$\ being the largest integer $\leq x$. For positive integers $c$ and
integers $b$ the classical Dedekind sum $s(b,c)$, arising in the theory of
Dedekind $\eta$--function, was introduced by R. Dedekind in 1892 as
\[
s(b,c)=\sum\limits_{j(\text{mod }c)}\left(  \left(  \frac{j}{c}\right)
\right)  \left(  \left(  \frac{bj}{c}\right)  \right)  .
\]
The most important property of Dedekind sums is the reciprocity
theorem\textbf{\ }%
\begin{equation}
s(b,c)+s(c,b)=-\frac{1}{4}+\frac{1}{12}\left(  \frac{b}{c}+\frac{c}{b}%
+\frac{1}{bc}\right)  \label{dr}%
\end{equation}
for $\gcd\left(  b,c\right)  =1$. For several proofs of (\ref{dr}) see
\cite{rg} (for recent studies on Dedekind sums the reader may consult to
\cite{kg,h,ks,lz,zh}). These sums were later generalized by various
mathematicians and the corresponding reciprocity laws were obtained. Apostol
\cite{ap} generalized $s(b,c)$ by defining
\[
s_{p}(b,c)=\sum\limits_{j=0}^{c-1}\overline{B}_{p}\left(  \frac{bj}{c}\right)
\overline{B}_{1}\left(  \frac{j}{c}\right)  ,
\]
where $\overline{B}_{p}\left(  x\right)  $ is the $p$th Bernoulli function
defined by
\[
\overline{B}_{n}\left(  x\right)  =B_{n}\left(  x-\left[  x\right]  \right)
\text{ for\ }n>1\text{\ and }\overline{B}_{1}(x)=\left(  \left(  x\right)
\right)  .
\]
Here $B_{n}\left(  x\right)  $ denotes the $n$th Bernoulli polynomial (see
Section 2). The sum $s_{p}(b,c)$ satisfies the following reciprocity formula
for odd $p$ and $\gcd\left(  b,c\right)  =1$ \cite{ap}:
\begin{equation}
(p+1)\left(  bc^{p}s_{p}\left(  b,c\right)  +\frac{{}}{{}}cb^{p}s_{p}\left(
c,b\right)  \right)  =\sum\limits_{j=0}^{p+1}\binom{p+1}{j}\left(  -1\right)
^{j}b^{j}c^{p+1-j}B_{p+1-j}B_{j}+pB_{p+1}. \label{dr1}%
\end{equation}

The another generalization is due to Berndt \cite{b2}. He gave a character
transformation formula similar to those for the Dedekind $\eta$--function and
defined Dedekind sums with character $s\left(  b,c:\chi\right)  $ by
\[
s\left(  b,c:\chi\right)  =\sum\limits_{n=0}^{ck-1}\chi\left(  n\right)
\overline{B}_{1,\chi}\left(  \frac{bn}{c}\right)  \overline{B}_{1}\left(
\frac{n}{ck}\right)  ,
\]
where $\chi$ denotes a non-principal primitive character modulo $k$ and
$\overline{B}_{p,\chi}\left(  x\right)  $ is the $p$th generalized Bernoulli
function, with $\overline{B}_{p,\chi}\left(  0\right)  =B_{p,\chi}$ (see
(\ref{4})). Using character transformation formula, he also derived the
following reciprocity law \cite[Theorem 4]{b2}%
\begin{equation}
s\left(  c,b:\chi\right)  +s\left(  b,c:\overline{\chi}\right)  =B_{1,\chi
}B_{1,\overline{\chi}}, \label{dkr}%
\end{equation}
whenever $b$ and $c$ are coprime positive integers, and either $c$ or
$b\equiv0$(mod $k$). For the proofs of (\ref{dr}) and (\ref{dkr}) via
(periodic) Poisson summation formula see \cite{b6}. The sum $s\left(
b,c:\chi\right)  $ is generalized by Cenkci et al \cite{cck} as
\begin{equation}
s_{p}\left(  b,c:\chi\right)  =\sum\limits_{n=0}^{ck-1}\chi\left(  n\right)
\overline{B}_{p,\chi}\left(  \frac{bn}{c}\right)  \overline{B}_{1}\left(
\frac{n}{ck}\right)  \label{5}%
\end{equation}
and the following reciprocity formula is established:

Let $p$\ be odd and $b,$ $c$\ be coprime positive integers. Let $\chi$\ be a
non-principal primitive character of modulus $k,$\ where $k$\ is a prime
number if $\gcd\left(  k,bc\right)  =1$, otherwise $k$\ is an arbitrary
integer. Then%
\begin{align}
&  (p+1)\left(  bc^{p}s_{p}(b,c:\chi)+cb^{p}s_{p}(c,b:\overline{\chi})\right)
\nonumber\\
&  \quad=\sum\limits_{j=0}^{p+1}\binom{p+1}{j}b^{j}c^{p+1-j}B_{j,\overline
{\chi}}B_{p+1-j,\chi}+\frac{p}{k}\chi\left(  c\right)  \overline{\chi}\left(
-b\right)  \left(  k^{p+1}-1\right)  B_{p+1}. \label{rp}%
\end{align}

Recently, the authors \cite{cd} have used the character analogue of the
Euler--MacLaurin summation formula in order to prove (\ref{rp}) (this method
was also exploited in \cite{ck,dc}). We perceive from this formula that the
sum $s_{p}\left(  b,c:\chi\right)  $\ can be generalized to sums such as
$s_{p}\left(  b,c:\chi_{1},\chi_{2}\right)  $ involving two primitive characters.

In this paper, we\textbf{ }systematically generalize $s_{p}\left(
b,c:\chi\right)  $ to sums involving two primitive characters and prove the
corresponding reciprocity formulas.

Firstly, we define character Dedekind sum involving primitive characters
$\chi_{1}$ and $\chi_{2}$ of modulus $k$ by%
\[
s_{p}\left(  b,c:\chi_{1},\chi_{2}\right)  =\sum\limits_{n=0}^{ck-1}\chi
_{1}\left(  n\right)  \overline{B}_{p,\chi_{2}}\left(  \frac{bn}{c}\right)
\overline{B}_{1}\left(  \frac{n}{ck}\right)  ,
\]
which is a natural generalization of the sum $s_{p}\left(  b,c:\chi\right)  $
given by (\ref{5}), i.e., $s_{p}\left(  b,c:\chi,\chi\right)  =s_{p}\left(
b,c:\chi\right)  $. Once again utilizing the power of the Euler--MacLaurin
summation formula we derive the following reciprocity formula.

\begin{theorem}
\label{rp1}Let $b,$ $c$ be positive integers with $q=\gcd\left(  b,c\right)  $
and $p>1$. Let $\chi_{1}$ and $\chi_{2}$ be non-principal primitive characters
of modulus $k$.\ For $\left(  -1\right)  ^{p+1}\chi_{1}(-1)\chi_{2}(-1)=1$ the
following reciprocity formula holds
\begin{align*}
&  \left(  p+1\right)  \left(  bc^{p}s_{p}\left(  b,c:\chi_{1},\chi
_{2}\right)  +cb^{p}s_{p}\left(  b,c:\overline{\chi_{2}},\overline{\chi_{1}%
}\right)  \right) \\
&  \ =\sum\limits_{j=0}^{p+1}\binom{p+1}{j}b^{j}c^{p+1-j}B_{j,\overline
{\chi_{1}}}B_{p+1-j,\chi_{2}}+pq^{p+1}k^{p-1}\sum\limits_{h=1}^{k-1}%
\sum\limits_{a=1}^{k-1}\chi_{1}(h)\overline{\chi}_{2}(a)\overline{B}%
_{p+1}\left(  \frac{ca+bh}{qk}\right)  .
\end{align*}

\end{theorem}

Secondly, we define the sum $\widehat{S}_{p}\left(  b,c:\chi_{1},\chi
_{2}\right)  $ for primitive characters $\chi_{1}$ and $\chi_{2}$ having
modulus $k_{1}$ and $k_{2},$ need not to be same, by
\[
\widehat{S}_{p}\left(  b,c:\chi_{1},\chi_{2}\right)  =\sum\limits_{n=0}%
^{ck_{1}k_{2}-1}\chi_{1}\left(  n\right)  \overline{B}_{p,\chi_{2}}\left(
\frac{nb}{c}\right)  \overline{B}_{1}\left(  \frac{n}{ck_{1}k_{2}}\right)  ,
\]
which reduces to $s_{p}\left(  b,c:\chi_{1},\chi_{2}\right)  $ for
$k_{1}=k_{2}.$

Finally, for primitive characters $\chi_{1}$ and $\chi_{2}$ having modulus
$k_{1}$ and $k_{2},$ we consider the following sum\textbf{ }%
\[
\widetilde{S}_{p}\left(  b,c:\chi_{1},\chi_{2}\right)  =\sum\limits_{n=0}%
^{ck_{1}-1}\chi_{1}\left(  n\right)  \overline{B}_{p,\chi_{2}}\left(
\frac{nbk_{2}}{ck_{1}}\right)  \overline{B}_{1}\left(  \frac{n}{ck_{1}%
}\right)  ,
\]
which generalizes the previous sums. More clearly, $\widetilde{S}_{p}\left(
b,c:\chi_{1},\chi_{2}\right)  =s_{p}\left(  b,c:\chi_{1},\chi_{2}\right)  $
for $k_{1}=k_{2},$ and $\widetilde{S}_{p}\left(  bk_{1},ck_{2}:\chi_{1}%
,\chi_{2}\right)  =\widehat{S}_{p}\left(  b,c:\chi_{1},\chi_{2}\right)  .$ We
obtain the following reciprocity formula for $\widetilde{S}_{p}\left(
b,c:\chi_{1},\chi_{2}\right)  .$

\begin{theorem}
\label{rp2}Let $b,$ $c$ be positive integers with $q=\gcd\left(  b,c\right)  $
and $p>1$. Let $\chi_{1}$ and $\chi_{2}$ be non-principal primitive characters
of modulus $k_{1}$ and $k_{2},$ respectively.\ For $\left(  -1\right)
^{p+1}\chi_{1}(-1)\chi_{2}(-1)=1$ the following reciprocity formula holds%
\begin{align*}
&  \left(  p+1\right)  \left(  bk_{2}\left(  ck_{1}\right)  ^{p}\text{
}\widetilde{S}_{p}\left(  b,c:\chi_{1},\chi_{2}\right)  +\frac{{}}{{}}%
ck_{1}\left(  bk_{2}\right)  ^{p}\text{ }\widetilde{S}_{p}\left(
c,b:\overline{\chi_{2}},\overline{\chi_{1}}\right)  \right) \\
&  \quad=\sum\limits_{j=0}^{p+1}\binom{p+1}{j}\left(  bk_{2}\right)
^{j}\left(  ck_{1}\right)  ^{p+1-j}B_{j,\overline{\chi_{1}}}B_{p+1-j,\chi_{2}%
}\\
&  \qquad+pq^{p+1}\left(  k_{1}k_{2}\right)  ^{p}\sum\limits_{h=1}^{k_{1}}%
\sum\limits_{j=1}^{k_{2}}\chi_{1}(h)\overline{\chi}_{2}(j)\overline{B}%
_{p+1}\left(  \frac{cj}{qk_{2}}+\frac{bh}{qk_{1}}\right)  .
\end{align*}

\end{theorem}

In particular we have the following corollary.

\begin{corollary}
\label{rp3}Let $b,$ $c$ be positive integers with $q=\gcd\left(  b,c\right)  $
and $p>1$. Let $\chi_{1}$ and $\chi_{2}$ be non-principal primitive characters
of modulus $k_{1}$ and $k_{2},$ respectively.\ For $\left(  -1\right)
^{p+1}\chi_{1}(-1)\chi_{2}(-1)=1$ and $k_{1}\not =k_{2},$ the following
reciprocity formula holds
\[
\left(  p+1\right)  \left(  bc^{p}\text{ }\widehat{S}_{p}\left(
b,c:\overline{\chi_{1}},\chi_{2}\right)  +\frac{{}}{{}}cb^{p}\text{
}\widehat{S}_{p}\left(  c,b:\overline{\chi_{2}},\chi_{1}\right)  \right)
=\sum\limits_{j=0}^{p+1}\binom{p+1}{j}c^{j}b^{p+1-j}B_{p+1-j,\chi_{1}%
}B_{j,\chi_{2}}.
\]

\end{corollary}

Moreover, we derive an explicit formula for the following type integral
\[
\int\limits_{0}^{x}B_{n_{1}}\left(  b_{1}z+y_{1}\right)  \cdots B_{n_{r}%
}\left(  b_{r}z+y_{r}\right)  dz
\]
(see Proposition \ref{th-i2}). Since this theorem is also valid for the Appell
polynomials, the earlier results given by Liu et al \cite{lpz}, Hu et al
\cite{hkk}, and Agoh and Dilcher \cite{ad} are direct consequences of the
derived formula (see Remark \ref{rem1}). As a consequence of this integral, we
have the following reciprocity relation for Bernoulli polynomials, which
generalizes \cite[Proposition 2]{ad}.

\begin{corollary}
\label{cor2}For all $n,m\geq0$ we have%
\begin{align}
&  \sum\limits_{a=0}^{n}\left(  -1\right)  ^{a}\binom{m+n+1}{n-a}b_{1}%
^{a}b_{2}^{-a-1}B_{n-a}\left(  b_{1}x+y_{1}\right)  B_{m+a+1}\left(
b_{2}x+y_{2}\right) \nonumber\\
&  \quad-\sum\limits_{a=0}^{m}\left(  -1\right)  ^{a}\binom{m+n+1}{m-a}%
b_{2}^{a}b_{1}^{-a-1}B_{m-a}\left(  b_{2}x+y_{2}\right)  B_{n+a+1}\left(
b_{1}x+y_{1}\right) \nonumber\\
&  =\frac{\left(  -1\right)  ^{m+1}}{b_{1}^{m+1}b_{2}^{n+1}}\sum
\limits_{a=0}^{m+n+1}\left(  -1\right)  ^{a}\binom{m+n+1}{a}b_{1}^{a}%
b_{2}^{m+n+1-a}B_{m+n+1-a}\left(  y_{1}\right)  B_{a}\left(  y_{2}\right)  ,
\label{24}%
\end{align}
where $b_{l}$ $\left(  b_{l}\not =0\right)  $ and $y_{l}$ $\left(  1\leq l\leq
r\right)  $ are real numbers.
\end{corollary}

\begin{remark}
\textbf{1)} If $\left(  n+m\right)  $ is even and $y_{1},$ $y_{2}$ are in
$\left\{  0,1/2,1\right\}  ,$ then the right-hand side of (\ref{24}) reduces
to
\[
\left(  -1\right)  ^{n}\frac{m+n+1}{b_{1}}\left\{  \left(  \frac{b_{2}}{b_{1}%
}\right)  ^{m-1}B_{m+n}\left(  y_{1}\right)  B_{1}\left(  y_{2}\right)
-\left(  \frac{b_{1}}{b_{2}}\right)  ^{n}B_{1}\left(  y_{1}\right)
B_{m+n}\left(  y_{2}\right)  \right\}  .
\]
\textbf{2) }If $\left(  n+m\right)  $ is odd and $y_{1}=$ $y_{2}=0$ or $1,$
then the right-hand side of (\ref{24}) is closely related to the reciprocity
formula of Apostol's Dedekind sums given by (\ref{dr1}). In fact, we have for
all $x$ and odd integer\textbf{ }$p=m+n,$ $\left(  n,m\geq0\right)  $\textbf{
}%
\begin{align*}
&  \left(  p+1\right)  \left(  b_{1}b_{2}^{p}s_{p}(b_{1},b_{2})+b_{2}b_{1}%
^{p}s_{p}(b_{2},b_{1})\right) \\
&  =\sum\limits_{a=0}^{n}\left(  -1\right)  ^{n-a}\binom{m+n+1}{n-a}%
b_{1}^{m+a+1}b_{2}^{n-a}B_{n-a}\left(  b_{1}x\right)  B_{m+a+1}\left(
b_{2}x\right) \\
&  \quad+\sum\limits_{a=0}^{m}\left(  -1\right)  ^{m-a}\binom{m+n+1}{m-a}%
b_{2}^{n+a+1}b_{1}^{m-a}B_{m-a}\left(  b_{2}x\right)  B_{n+a+1}\left(
b_{1}x\right)  +q^{p+1}pB_{p+1}\\
&  =\sum\limits_{a=0}^{m+n+1}\left(  -1\right)  ^{a}\binom{m+n+1}{a}b_{1}%
^{a}b_{2}^{m+n+1-a}B_{m+n+1-a}B_{a}+q^{p+1}pB_{p+1},
\end{align*}
where $q=\gcd\left(  b_{1},b_{2}\right)  .$
\end{remark}

We further derive the Laplace transform of $\overline{B}_{n}\left(
tu+y\right)  $ (see (\ref{16})), which coincides with \cite[Lemma 4]{ml} for
$t=1$ and $y\in\mathbb{Z}$. Similar results are also obtained for the
generalized Bernoulli polynomials.

We summarize this study as follows: Section 2 is the preliminary section where
we give definitions and known results needed. In Section 3, we prove the
reciprocity formulas for character Dedekind sums by using the character
analogue of the Euler--MacLaurin summation formula. In Section 4, we first
derive a formula for the integral having more general integrands. By this, we
extend known results on the integral of products of Appell polynomials\ and
illustrate a formula for the integral of products of Bernoulli polynomials.
Furthermore, the proof of Corollary \ref{cor2} and the Laplace transform of
$\overline{B}_{n}\left(  tu+y\right)  $ are given.

\section{Preliminaries}

The Bernoulli polynomials $B_{n}(x)$ are defined by means of the generating
function
\begin{equation}
\frac{te^{xt}}{e^{t}-1}=\sum\limits_{n=0}^{\infty}B_{n}(x)\frac{t^{n}}%
{n!}\quad\left(  \left\vert t\right\vert <2\pi\right)  \label{0}%
\end{equation}
and $B_{n}=B_{n}(0)$ are the Bernoulli numbers with $B_{0}=1,$ $B_{1}=-1/2$
and $B_{2n+1}=B_{2n-1}\left(  1/2\right)  =0$ for $n\geq1.$

Let $\chi$ be a primitive character of modulus $k.$ The generalized Bernoulli
polynomials $B_{n,\chi}\left(  x\right)  $ are defined by means of the
generating function \cite{b5}
\[
\sum\limits_{a=0}^{k-1}\frac{\overline{\chi}(a)te^{(a+x)t}}{e^{kt}-1}%
=\sum\limits_{n=0}^{\infty}B_{n,\chi}\left(  x\right)  \frac{t^{n}}{n!}%
\quad\left(  \left\vert t\right\vert <2\pi/k\right)
\]
and $B_{n,\chi}=B_{n,\chi}\left(  0\right)  $ are the generalized Bernoulli
numbers. In particular, if $\chi_{0}$ is the principal character, then
$B_{n,\chi_{0}}\left(  x\right)  =B_{n}\left(  x\right)  $ for $n\geq0$\ and
$B_{0,\chi}\left(  x\right)  =0$ for $\chi\not =\chi_{0}.$ The generalized
Bernoulli functions $\overline{B}_{n,\chi}\left(  x\right)  ,$ are functions
with period $k$, may be defined by (\cite[Theorem 3.1]{b5})
\begin{equation}
\overline{B}_{m,\chi}\left(  x\right)  =k^{m-1}\sum\limits_{n=1}%
^{k-1}\overline{\chi}(n)\overline{B}_{m}\left(  \frac{n+x}{k}\right)  ,\text{
}m\geq1, \label{4}%
\end{equation}
for all real $x$. We\ recall some properties that we need in the sequel.%
\begin{align}
&  \frac{d}{dx}B_{m}\left(  x\right)  =mB_{m-1}\left(  x\right)  \text{ and
}\frac{d}{dx}B_{m,\chi}\left(  x\right)  =mB_{m-1,\chi}\left(  x\right)
,\text{ }m\geq1,\label{1}\\
&  \frac{d}{dx}\overline{B}_{m,\chi}\left(  x\right)  =m\overline{B}%
_{m-1,\chi}\left(  x\right)  ,\text{ }m\geq2\text{ and }\overline{B}_{m,\chi
}\left(  k\right)  =\overline{B}_{m,\chi}\left(  0\right)  =B_{m,\chi
},\label{8}\\
&  B_{m,\chi}\left(  -x\right)  =\left(  -1\right)  ^{m}\chi\left(  -1\right)
B_{m,\chi}\left(  x\right)  ,\text{ }m\geq0. \label{3}%
\end{align}
It is also known that the degree of Bernoulli polynomials $B_{n}\left(
x\right)  $\ is $n,$ and the degree of generalized Bernoulli polynomials
$B_{n,\chi}\left(  x\right)  $\ is not greater than $n-1.$

We also need the character analogue of the Euler--MacLaurin summation formula,
due to Berndt \cite{b5}, which is presented here in the following form.

\begin{theorem}
(\cite[Theorem 4.1]{b5})\label{E-M} Let $f\in C^{(l+1)}\left[  \alpha
,\beta\right]  ,$ $-\infty<\alpha<\beta<\infty.$ Then,%
\begin{align*}
\sum_{\alpha\leq n\leq\beta}\hspace{-0.05in}^{^{\prime}}\chi\left(  n\right)
f(n)  &  =\chi\left(  -1\right)  \sum\limits_{j=0}^{l}\frac{\left(  -1\right)
^{j+1}}{(j+1)!}\left(  \overline{B}_{j+1,\overline{\chi}}\left(  \beta\right)
f^{(j)}(\beta)-\overline{B}_{j+1,\overline{\chi}}(\alpha)f^{(j)}%
(\alpha)^{\text{\ }}\right) \\
&  \quad+\chi\left(  -1\right)  \frac{(-1)^{l}}{(l+1)!}\int\limits_{\alpha
}^{\beta}\overline{B}_{l+1,\overline{\chi}}\left(  u\right)  f^{(l+1)}(u)du.
\end{align*}
where the dash indicates that if $n=\alpha$ or $n=\beta$, then only $\frac
{1}{2}\chi(\alpha)f(\alpha)$ or $\frac{1}{2}\chi(\beta)f(\beta)$ is counted, respectively.
\end{theorem}

\section{Proofs of Theorems \ref{rp1} and \ref{rp2}}

We recall the sum $s_{p}\left(  b,c:\chi_{1},\chi_{2}\right)  $ defined by
\[
s_{p}\left(  b,c:\chi_{1},\chi_{2}\right)  =\sum\limits_{n=0}^{ck-1}\chi
_{1}\left(  n\right)  \overline{B}_{p,\chi_{2}}\left(  \frac{bn}{c}\right)
\overline{B}_{1}\left(  \frac{n}{ck}\right)  ,
\]
where $\chi_{1}$ and $\chi_{2}$ are primitive characters of modulus $k.$ In
view of (\ref{3}) it is easy to see that
\begin{equation}
s_{p}\left(  b,c:\chi_{1},\chi_{2}\right)  =\left(  -1\right)  ^{p+1}\chi
_{1}(-1)\chi_{2}(-1)s_{p}\left(  b,c:\chi_{1},\chi_{2}\right)  , \label{10}%
\end{equation}
which entails $s_{p}\left(  b,c:\chi_{1},\chi_{2}\right)  =0$ for $\left(
-1\right)  ^{p+1}\chi_{1}(-1)\chi_{2}(-1)=-1.$

\begin{proof}
[Proof of Theorem \ref{rp1}]With the aid of $\overline{B}_{1}(x)=x-1/2$ for
$0<x<1$, we have%
\begin{equation}
cks_{p}\left(  b,c:\chi_{1},\chi_{2}\right)  =\sum\limits_{n=0}^{ck-1}\chi
_{1}\left(  n\right)  n\overline{B}_{p,\chi_{2}}\left(  \frac{bn}{c}\right)
-\frac{ck}{2}\sum\limits_{n=0}^{ck-1}\chi_{1}\left(  n\right)  \overline
{B}_{p,\chi_{2}}\left(  \frac{bn}{c}\right)  . \label{7}%
\end{equation}
Thus, let $f(x)=x\overline{B}_{p,\chi_{2}}\left(  xb/c\right)  ,$ $\alpha=0$
and $\beta=ck$\ in Theorem \ref{E-M} and let $p>1$. Equation (\ref{8}) entails
that $f\in C^{\left(  p-1\right)  }\left[  \alpha,\beta\right]  $ and
\begin{equation}
\frac{d^{j}}{dx^{j}}f\left(  x\right)  =\frac{p!}{\left(  p-j\right)
!}\left(  \frac{b}{c}\right)  ^{j}x\overline{B}_{p-j,\chi_{2}}\left(  \frac
{b}{c}x\right)  +j\frac{p!}{\left(  p+1-j\right)  !}\left(  \frac{b}%
{c}\right)  ^{j-1}\overline{B}_{p+1-j,\chi_{2}}\left(  \frac{b}{c}x\right)
\label{12}%
\end{equation}
for $0\leq j\leq p-1.$ Therefore with the help of (\ref{8}) and Theorem
\ref{E-M} we have
\begin{align}
\sum_{n=1}^{ck-1}\chi_{1}\left(  n\right)  n\overline{B}_{p,\chi_{2}}\left(
n\frac{b}{c}\right)   &  =\chi_{1}\left(  -1\right)  \frac{ck}{p+1}%
\sum\limits_{j=0}^{l}\left(  -1\right)  ^{j+1}\binom{p+1}{j+1}\left(  \frac
{b}{c}\right)  ^{j}B_{j+1,\overline{\chi_{1}}}B_{p-j,\chi_{2}}\nonumber\\
&  +\chi_{1}\left(  -1\right)  \binom{p}{l+1}\left(  -\frac{b}{c}\right)
^{l}bc\int\limits_{0}^{k}x\overline{B}_{l+1,\overline{\chi_{1}}}\left(
cx\right)  \overline{B}_{p-(l+1),\chi_{2}}\left(  bx\right)  dx\nonumber\\
&  +\chi_{1}\left(  -1\right)  \left(  -1\right)  ^{l}\binom{p}{l}\left(
\frac{b}{c}\right)  ^{l}c\int\limits_{0}^{k}\overline{B}_{l+1,\overline
{\chi_{1}}}\left(  cx\right)  \overline{B}_{p-l,\chi_{2}}\left(  bx\right)  dx
\label{13}%
\end{align}
for $1\leq l+1\leq p-1.$\ On the other hand, applying character analogue of
the Euler--MacLaurin summation formula to the generalized Bernoulli function
$\overline{B}_{p+1,\chi_{2}}\left(  bx/c\right)  $ gives%
\begin{equation}
\sum\limits_{n=1}^{ck-1}\chi_{1}\left(  n\right)  \overline{B}_{p+1,\chi_{2}%
}\left(  \frac{bn}{c}\right)  =\chi_{1}\left(  -1\right)  \left(  -1\right)
^{l}\binom{p+1}{l+1}\left(  \frac{b}{c}\right)  ^{l+1}c\int\limits_{0}%
^{k}\overline{B}_{l+1,\overline{\chi_{1}}}\left(  cx\right)  \overline
{B}_{p-l,\chi_{2}}\left(  bx\right)  dx. \label{11}%
\end{equation}

Let $\left(  -1\right)  ^{p+1}\chi_{1}(-1)\chi_{2}(-1)=1.$ Taking into account
that
\begin{equation}
\sum\limits_{n=1}^{ck-1}\chi_{1}\left(  n\right)  \overline{B}_{p,\chi_{2}%
}\left(  \frac{bn}{c}\right)  =0 \label{lek2-a}%
\end{equation}
for $\left(  -1\right)  ^{p+1}\chi_{1}(-1)\chi_{2}(-1)=1$, it follows from
(\ref{7}), (\ref{13}) and (\ref{11}) that%
\begin{align}
cks_{p}\left(  b,c:\chi_{1},\chi_{2}\right)   &  =\chi_{1}\left(  -1\right)
\frac{ck}{p+1}\sum\limits_{j=0}^{l}\left(  -1\right)  ^{j+1}\binom{p+1}%
{j+1}\left(  \frac{b}{c}\right)  ^{j}B_{j+1,\overline{\chi_{1}}}%
B_{p-j,\chi_{2}}\nonumber\\
&  +\chi_{1}\left(  -1\right)  \binom{p}{l+1}\left(  -\frac{b}{c}\right)
^{l}bc\int\limits_{0}^{k}x\overline{B}_{l+1,\overline{\chi_{1}}}\left(
cx\right)  \overline{B}_{p-(l+1),\chi_{2}}\left(  bx\right)  dx\nonumber\\
&  +\frac{l+1}{p+1}\frac{c}{b}\sum\limits_{n=1}^{ck-1}\chi_{1}\left(
n\right)  \overline{B}_{p+1,\chi_{2}}\left(  \frac{bn}{c}\right)  . \label{2}%
\end{align}
We first set $l+1=p-1$ in (\ref{2}) and then multiply by $bc^{p-1}/k.$ So we
have%
\begin{align}
bc^{p}s_{p}\left(  b,c:\chi_{1},\chi_{2}\right)   &  =\frac{\chi_{1}\left(
-1\right)  }{p+1}\sum\limits_{j=1}^{p-1}\left(  -1\right)  ^{j}\binom{p+1}%
{j}b^{j}c^{p+1-j}B_{j,\overline{\chi_{1}}}B_{p+1-j,\chi_{2}}\nonumber\\
&  +\left(  -1\right)  ^{p}\chi_{1}\left(  -1\right)  \frac{pb^{p}c^{2}}%
{k}\int\limits_{0}^{k}x\overline{B}_{p-1,\overline{\chi_{1}}}\left(
cx\right)  \overline{B}_{1,\chi_{2}}\left(  bx\right)  dx\nonumber\\
&  +\frac{p-1}{p+1}\frac{c^{p}}{k}\sum\limits_{n=1}^{ck-1}\chi_{1}\left(
n\right)  \overline{B}_{p+1,\chi_{2}}\left(  \frac{bn}{c}\right)  . \label{14}%
\end{align}
Now we first set $l=0,$ interchange $b$ and $c$ and also $\chi_{1}$ and
$\overline{\chi_{2}}$ \ in (\ref{2}), and then multiply by $cb^{p-1}/k.$ Thus
we have%
\begin{align}
cb^{p}s_{p}\left(  c,b:\overline{\chi_{2}},\overline{\chi_{1}}\right)   &
=-\chi_{2}\left(  -1\right)  cb^{p}B_{p,\overline{\chi_{1}}}B_{1,\chi_{2}%
}\nonumber\\
&  +\overline{\chi_{2}}\left(  -1\right)  \frac{pb^{p}c^{2}}{k}\int%
\limits_{0}^{k}x\overline{B}_{p-1,\overline{\chi_{1}}}\left(  cx\right)
\overline{B}_{1,\chi_{2}}\left(  bx\right)  dx\nonumber\\
&  +\frac{1}{p+1}\frac{b^{p}}{k}\sum\limits_{n=1}^{bk-1}\overline{\chi_{2}%
}\left(  n\right)  \overline{B}_{p+1,\overline{\chi_{1}}}\left(  \frac{cn}%
{b}\right)  . \label{15}%
\end{align}
We now consider the sum
\[
\sum\limits_{n=1}^{ck-1}\chi_{1}(n)\overline{B}_{p+1,\chi_{2}}\left(
\frac{bn}{c}\right)  .
\]
As mentioned in (\ref{lek2-a}) the sum vanishes when $\left(  -1\right)
^{p+1}\chi_{1}(-1)\chi_{2}(-1)=-1.$ For $\left(  -1\right)  ^{p+1}\chi
_{1}(-1)\chi_{2}(-1)=1$ and $\gcd(b,c)=1,$\ first setting $n=h+mk,$ where
$1\leq h\leq k-1,$ $0\leq m\leq c-1,$ and then using (\ref{4}) and Raabe
formula
\[
\sum\limits_{m=0}^{c-1}\overline{B}_{p+1}\left(  \frac{m+x}{c}\right)
=c^{-p}\overline{B}_{p+1}\left(  x\right)  ,
\]
we deduce that%
\begin{equation}
\sum\limits_{n=1}^{ck-1}\chi_{1}(n)\overline{B}_{p+1,\chi_{2}}\left(
\frac{bn}{c}\right)  =\left(  \frac{k}{c}\right)  ^{p}\sum\limits_{h=1}%
^{k-1}\sum\limits_{j=1}^{k-1}\chi_{1}(h)\overline{\chi}_{2}(j)\overline
{B}_{p+1}\left(  \frac{cj}{k}+\frac{bh}{k}\right)  . \label{lek2}%
\end{equation}
We mention also that if $\gcd(bq,cq)=q,$ then
\[
\sum\limits_{n=1}^{qck-1}\chi_{1}(n)\overline{B}_{p+1,\chi_{2}}\left(
\frac{qbn}{qc}\right)  =q\sum\limits_{n=1}^{ck-1}\chi_{1}(n)\overline
{B}_{p+1,\chi_{2}}\left(  \frac{bn}{c}\right)  .
\]
Therefore, combining (\ref{14}) and (\ref{15}), with the use of $\left(
-1\right)  ^{p+1}\chi_{1}(-1)\chi_{2}(-1)=1,$ (\ref{3}) and (\ref{lek2}), we
obtain the reciprocity formula.
\end{proof}

\begin{remark}
The sum occurs in (\ref{lek2}) was evaluated in \cite[Lemma 5.5]{cck} when
$\chi_{1}=\chi_{2}.$ In that case, the right-hand side of (\ref{lek2}%
)\ becomes $c^{1-p}\chi\left(  c\right)  \overline{\chi}\left(  -b\right)
\left(  k^{p}-1\right)  \overline{B}_{p}\left(  0\right)  $ under the
condition that $k$ is a prime number when $\gcd(k,bc)=1$ and $p$ is even,
otherwise $k$ is an arbitrary integer. By similar method, this sum can be
evaluated when $\chi_{1}\not =\chi_{2}$, in which case we have
\[
\sum\limits_{n=1}^{ck-1}\chi_{1}(n)\overline{B}_{p,\chi_{2}}\left(  \frac
{bn}{c}\right)  =-\epsilon\frac{p!}{(2\pi i)^{p}}\left(  \frac{k}{c}\right)
^{p-1}G(b,\chi_{1})G(c,\overline{\chi_{2}})L\left(  p,\overline{\chi_{1}}%
\chi_{2}\right)
\]
for $\gcd(b,c)=1$ with $c>0.$ Here $\epsilon=1+\left(  -1\right)  ^{p}\chi
_{1}(-1)\chi_{2}(-1)$ and $\overline{\chi_{1}}\chi_{2}$ is a Dirichlet
character modulo $k$, and $G(b,\chi)$ and $L\left(  p,\overline{\chi_{1}}%
\chi_{2}\right)  $ stands for the Gauss sum and the Dirichlet L--function,
respectively. However, the statement of reciprocity formula must be given
separately for the cases $\chi_{1}=\chi_{2}$ and $\chi_{1}\not =\chi_{2}.$
\end{remark}

\begin{proof}
[Proof of Theorem \ref{rp2}]We first note that similar to $s_{p}\left(
b,c:\chi_{1},\chi_{2}\right)  $ the sum $\widetilde{S}_{p}\left(  b,c:\chi
_{1},\chi_{2}\right)  $ vanishes when $\left(  -1\right)  ^{p+1}\chi
_{1}(-1)\chi_{2}(-1)=-1.$ Let $\gcd(b,c)=1$ with $c>0.$ Then the counterpart
of (\ref{lek2}) becomes%
\begin{equation}
\sum\limits_{n=1}^{ck_{1}}\chi_{1}(n)\overline{B}_{p+1,\chi_{2}}\left(
\frac{nbk_{2}}{ck_{1}}\right)  =\left(  \frac{k_{2}}{c}\right)  ^{p}%
\sum\limits_{h=1}^{k_{1}}\sum\limits_{j=1}^{k_{2}}\chi_{1}(h)\overline{\chi
}_{2}(j)\overline{B}_{p+1}\left(  \frac{cj}{k_{2}}+\frac{bh}{k_{1}}\right)
\label{lek3}%
\end{equation}
when $\left(  -1\right)  ^{p+1}\chi_{1}(-1)\chi_{2}(-1)=1,$ and this sum
vanishes when $\left(  -1\right)  ^{p+1}\chi_{1}(-1)\chi_{2}(-1)=-1.$ Now let
$f(x)=x\overline{B}_{p,\chi_{2}}\left(  xbk_{2}/ck_{1}\right)  ,$ $\alpha=0,$
$\beta=ck$ in Theorem \ref{E-M} and $p>1$. Then adopting the arguments in the
proof of Theorem \ref{rp1}, and using (\ref{lek3}) rather than (\ref{lek2}%
),\textbf{ }the desired result follows.
\end{proof}

\subsection{Further consequences}

Let $\left(  -1\right)  ^{p+1}\chi_{1}(-1)\chi_{2}(-1)=1.$ If we take $c=1$
and $b=k$ in (\ref{11}). Using (\ref{8}) and then the fact $\sum_{n=1}%
^{k-1}\chi\left(  n\right)  =0,$ we find that
\[
\int\limits_{0}^{k}\overline{B}_{l+1,\overline{\chi_{1}}}\left(  x\right)
\overline{B}_{p-l,\chi_{2}}\left(  kx\right)  dx=0.
\]
It is also obvious from (\ref{11}), for $b=c=1,$
\[
\chi_{1}\left(  -1\right)  \left(  -1\right)  ^{l+1}\binom{p+1}{l+1}%
\int\limits_{0}^{k}\overline{B}_{l+1,\overline{\chi_{1}}}\left(  x\right)
\overline{B}_{p-l,\chi_{2}}\left(  x\right)  dx=\sum\limits_{n=1}^{k}\chi
_{1}\left(  n\right)  \overline{B}_{p+1,\chi_{2}}\left(  n\right)  .
\]

Now let $\left(  -1\right)  ^{p+1}\chi_{1}(-1)\chi_{2}(-1)=-1.$ Then it
follows from (\ref{11}) and (\ref{lek2-a}) that
\begin{equation}
\int\limits_{0}^{k}\overline{B}_{l+1,\overline{\chi_{1}}}\left(  cx\right)
\overline{B}_{p-l,\chi_{2}}\left(  bx\right)  dx=0. \label{20}%
\end{equation}
It is seen from (\ref{10}) and (\ref{7}) that
\[
\sum\limits_{n=1}^{ck-1}\chi_{1}\left(  n\right)  n\overline{B}_{p,\chi_{2}%
}\left(  \frac{bn}{c}\right)  =\frac{ck}{2}\sum\limits_{n=1}^{ck-1}\chi
_{1}\left(  n\right)  \overline{B}_{p,\chi_{2}}\left(  \frac{bn}{c}\right)  .
\]
Since $B_{p+1-j,\chi_{2}}=\left(  -1\right)  ^{p+1-j}\chi_{2}\left(
-1\right)  B_{p+1-j,\chi_{2}}$ from (\ref{3}), we conclude from (\ref{13}),
(\ref{lek2}) and (\ref{20}) that%
\begin{align*}
&  \binom{p}{l+1}\left(  -\frac{b}{c}\right)  ^{l}b\int\limits_{0}%
^{k}x\overline{B}_{l+1,\overline{\chi_{1}}}\left(  cx\right)  \overline
{B}_{p-(l+1),\chi_{2}}\left(  bx\right)  dx\\
&  \quad=\chi_{1}\left(  -1\right)  \frac{k}{2}\left(  \frac{k}{c}\right)
^{p-1}\sum\limits_{h=1}^{k-1}\sum\limits_{j=1}^{k-1}\chi_{1}(h)\overline{\chi
}_{2}(j)\overline{B}_{p}\left(  \frac{cj}{k}+\frac{bh}{k}\right)
\end{align*}
for $\gcd\left(  b,c\right)  =1.$ In particular%
\[
\int\limits_{0}^{k}x\overline{B}_{l+1,\overline{\chi_{1}}}\left(  x\right)
\overline{B}_{p-(l+1),\chi_{2}}\left(  kx\right)  dx=0.
\]

\section{Integral of products of Bernoulli polynomials}

\subsection{Bernoulli polynomials}

In \cite{ad}, Agoh and Dilcher derived an explicit formula for the integral of
product of three Bernoulli polynomials by considering the interval of
integration $\left[  0,x\right]  ,$ rather than $\left[  0,1\right]  $ as in
\cite{c,em,mi,mo,n,w}, which are thus special cases of the following.

\begin{proposition}
(\cite[Proposition 3]{ad})\label{prop1} For all $l,m,n\geq0$ we have%
\begin{align*}
&  \frac{1}{l!m!n!}\int\limits_{0}^{x}B_{l}\left(  z\right)  B_{m}\left(
z\right)  B_{n}\left(  z\right)  dz\\
&  \quad=\sum\limits_{a=0}^{l+m}\left(  -1\right)  ^{a}\sum\limits_{j=0}%
^{a}\binom{a}{j}\frac{B_{l-a+j}\left(  x\right)  B_{m-j}\left(  x\right)
B_{n+a+1}\left(  x\right)  -B_{l-a+j}B_{m-j}B_{n+a+1}}{\left(  l-a+j\right)
!\left(  m-j\right)  !\left(  n+a+1\right)  !}.
\end{align*}

\end{proposition}

Recently, this has been extended by Hu et al \cite{hkk} to integral of product
of $r$ Bernoulli polynomials. Let $r$ be any positive integer. The multinomial
coefficients $\binom{\mu}{n_{1},...,n_{r}}$ are defined by%
\[
\binom{\mu}{n_{1},...,n_{r}}=\frac{\mu!}{n_{1}!\cdots n_{r}!},
\]
where $n_{1}+\cdots+n_{r}=\mu$ and $n_{1},...,n_{r}\geq0.$ Let
\begin{align}
I_{n_{1},...,n_{r}}(x)  &  =\frac{1}{n_{1}!\cdots n_{r}!}\int\limits_{0}%
^{x}B_{n_{1}}\left(  z\right)  \cdots B_{n_{r}}\left(  z\right)
dz,\label{19}\\
C_{n_{1},...,n_{r}}(x)  &  =\frac{1}{n_{1}!\cdots n_{r}!}\left(  B_{n_{1}%
}\left(  x\right)  \cdots B_{n_{r}}\left(  x\right)  -B_{n_{1}}\cdots
B_{n_{r}}\right)  .\nonumber
\end{align}
Hu et al \cite{hkk} derive the following theorem.

\begin{theorem}
\label{th-i1}(\cite[Theorem 1.5]{hkk}) For any $n_{1},\ldots,n_{r}\geq0$, we
have%
\[
I_{n_{1},...,n_{r}}(x)=\sum\limits_{a=0}^{n_{1}+\cdots+n_{r-1}}\left(
-1\right)  ^{a}\sum\limits_{j_{1}+\cdots+j_{r-1}=a}\binom{a}{j_{1}%
,...,j_{r-1}}C_{n_{1}-j_{1},\ldots,n_{r-1}-j_{r-1},n_{r}+a+1}(x).
\]

\end{theorem}

More recently, this result has been extended to the Appell polynomials (such
polynomials satisfying $\dfrac{d}{dz}A_{n}(z)=nA_{n-1}(z),$ $n=0,1,2,...$) by
Liu et al \cite[Theorem 1.1]{lpz}.

We can infer from (\ref{30}) and (\ref{th-i}) (see below) that these results
are valid in a more general form and can be proved more easily. We take
advantage of the following property of derivative
\begin{equation}
\left(  f_{1}\left(  z\right)  \cdots f_{m}\left(  z\right)  \right)
^{\left(  a\right)  }=\sum\limits_{j_{1}+\cdots+j_{m}=a}\binom{a}%
{j_{1},...,j_{m}}f_{1}{}^{\left(  j_{1}\right)  }\left(  z\right)  \cdots
f_{m}^{\left(  j_{m}\right)  }\left(  z\right)  , \label{30}%
\end{equation}
which can be easily seen by induction on $m.$

Let $f(z)$ be differentiable of order $\mu+1,$ and $P_{n}\left(  z\right)  $
be a polynomial (or function) such that $dP_{n}(z)/dz=nP_{n-1}(z).$
Integration by parts gives%
\[
\frac{1}{n!}\int\limits_{0}^{x}f(z)P_{n}\left(  z\right)  dz=\frac{1}{\left(
n+1\right)  !}\left[  f(z)P_{n+1}\left(  z\right)  \right]  _{0}^{x}-\frac
{1}{\left(  n+1\right)  !}\int\limits_{0}^{x}f\ ^{\prime}(z)P_{n+1}\left(
z\right)  dz.
\]
Using $\mu$ additional integrations by parts it is seen that
\begin{align}
\frac{1}{n!}\int\limits_{0}^{x}f(z)P_{n}\left(  z\right)  dz  &
=\sum\limits_{a=0}^{\mu}\frac{\left(  -1\right)  ^{a}}{\left(  n+a+1\right)
!}\left[  f^{\left(  a\right)  }\left(  z\right)  P_{n+a+1}\left(  z\right)
^{\ }\right]  _{0}^{x}\nonumber\\
&  +\frac{\left(  -1\right)  ^{\mu+1}}{\left(  n+\mu+1\right)  !}%
\int\limits_{0}^{x}f\ ^{\left(  \mu+1\right)  }(z)P_{n+\mu+1}\left(  z\right)
dz. \label{th-i}%
\end{align}
In particular if $f^{\left(  \mu+1\right)  }(z)=0,$ then%
\[
\frac{1}{n!}\int\limits_{0}^{x}f(z)P_{n}\left(  z\right)  dz=\sum
\limits_{a=0}^{\mu}\frac{\left(  -1\right)  ^{a}}{\left(  n+a+1\right)
!}\left(  f^{\left(  a\right)  }\left(  x\right)  P_{n+a+1}\left(  x\right)
-f^{\left(  a\right)  }\left(  0\right)  P_{n+a+1}\left(  0\right)  \frac{{}%
}{{}}\right)  .
\]

\begin{remark}
\label{rem1}By considering
\[
f\left(  z\right)  =B_{n_{1}}\left(  z\right)  \cdots B_{n_{r-1}}\left(
z\right)  \text{ and }P_{n}\left(  z\right)  =B_{n_{r}}\left(  z\right)
\]
and for an Appell polynomial $A_{n}\left(  z\right)  $%
\[
f\left(  z\right)  =A_{n_{1}}\left(  z\right)  \cdots A_{n_{r-1}}\left(
z\right)  \text{ and }P_{n}\left(  z\right)  =A_{n_{r}}\left(  z\right)
\]
in (\ref{th-i}), and using (\ref{30}) we have Hu et al's \cite{hkk} and Liu et
al's \cite{lpz} results, respectively.
\end{remark}

Let $b_{l}$ $\left(  b_{l}\not =0\right)  $ and $y_{l}$ $\left(  1\leq l\leq
r\right)  $ be arbitrary real numbers and let
\begin{align*}
\widehat{I}_{n_{1},...,n_{r}}(x;b;y)  &  =\widehat{I}_{n_{1},...,n_{r}%
}(x;b_{1},...,b_{r};y_{1},...,y_{r})\\
&  =\frac{1}{n_{1}!\cdots n_{r}!}\int\limits_{0}^{x}B_{n_{1}}\left(
b_{1}z+y_{1}\right)  \cdots B_{n_{r}}\left(  b_{r}z+y_{r}\right)  dz,\\
\widehat{C}_{n_{1},...,n_{r}}(x;b;y)  &  =\widehat{C}_{n_{1},...,n_{r}}\left(
x;b_{1},...,b_{r};y_{1},...,y_{r}\right) \\
&  =\frac{1}{n_{1}!\cdots n_{r}!}\left(  \prod\limits_{l=1}^{r}B_{n_{l}%
}\left(  b_{l}x+y_{l}\right)  -\prod\limits_{l=1}^{r}B_{n_{l}}\left(
y_{l}\right)  \right)  .
\end{align*}
We relate $\widehat{I}_{n_{1},...,n_{r}}(x;b;y)$ and $\widehat{C}%
_{n_{1},...,n_{r}}(x;b;y)$ in the following proposition.

\begin{proposition}
\label{th-i2}For any $n_{1},\ldots,n_{r}\geq0$, we have%
\begin{align}
\widehat{I}_{n_{1},...,n_{r}}(x;b;y)  &  =\sum\limits_{a=0}^{n_{1}%
+\cdots+n_{r-1}}\left(  -1\right)  ^{a}\sum\limits_{j_{1}+\cdots+j_{r-1}%
=a}\binom{a}{j_{1},...,j_{r-1}}\label{32}\\
&  \times b_{1}^{j_{1}}\cdots b_{r-1}^{j_{r-1}}b_{r}^{-a-1}\widehat{C}%
_{n_{1}-j_{1},\ldots,n_{r-1}-j_{r-1},n_{r}+a+1}(x;b;y).\nonumber
\end{align}

\end{proposition}

\begin{proof}
Setting
\[
f\left(  z\right)  =B_{n_{1}}\left(  b_{1}z+y_{1}\right)  \cdots B_{n_{r-1}%
}\left(  b_{r-1}z+y_{r-1}\right)  \text{ and }P_{n}\left(  z\right)
=B_{n_{r}}\left(  b_{r}z+y_{r}\right)
\]
in (\ref{th-i}) for $\mu=n_{1}+\cdots+n_{r-1},$ then using (\ref{1}) and
(\ref{30}) we have the desired result.
\end{proof}

Similar to \cite[Corollary 2]{ad} and \cite[Corollary 1.7]{hkk}, it is seen
from the definition of the integral $\widehat{I}_{n_{1},...,n_{r}}(x;b;y)$
that the right-hand side of (\ref{32}) is invariant under all permutations.

\begin{corollary}
\label{cor1}Let $T_{n_{1},...,n_{r}}(x;b;y)=T_{n_{1},...,n_{r}}(x;b_{1}%
,...,b_{r};y_{1},...,y_{r})$ be the right-hand side of (\ref{32}), and let
$\sigma\in S_{r}$, where $S_{r}$ is the symmetric group of degree $r$. Then
for all $n_{1},\ldots,n_{r}\geq0,$%
\[
T_{n_{1},...,n_{r}}(x;b;y)=T_{\sigma\left(  n_{1}\right)  ,...,\sigma\left(
n_{r}\right)  }(x;b_{\sigma};y_{\sigma}),
\]
where $b_{\sigma}=\left(  b_{\sigma\left(  n_{1}\right)  },...,b_{\sigma
\left(  n_{r}\right)  }\right)  $ and $y_{\sigma}=\left(  y_{\sigma\left(
n_{1}\right)  },...,y_{\sigma\left(  n_{r}\right)  }\right)  .$
\end{corollary}

\begin{proof}
[\textbf{Proof of Corollary }\ref{cor2}]From Corollary \ref{cor1} and
Proposition \ref{th-i2} for $r=2,$ we have%
\begin{align}
&  \sum\limits_{a=0}^{n}\left(  -1\right)  ^{a}\binom{m+n+1}{n-a}b_{1}%
^{a}b_{2}^{-a-1}\left(  B_{n-a}\left(  b_{1}x+y_{1}\right)  B_{m+a+1}\left(
b_{2}x+y_{2}\right)  -B_{n-a}\left(  y_{1}\right)  B_{m+a+1}\left(
y_{2}\right)  \right) \nonumber\\
&  =\sum\limits_{a=0}^{m}\left(  -1\right)  ^{a}\binom{m+n+1}{m-a}b_{2}%
^{a}b_{1}^{-a-1}\left(  B_{m-a}\left(  b_{2}x+y_{2}\right)  B_{n+a+1}\left(
b_{1}x+y_{1}\right)  -B_{m-a}\left(  y_{2}\right)  B_{n+a+1}\left(
y_{1}\right)  \right)  . \label{25}%
\end{align}
Let
\begin{align*}
T  &  :=\sum\limits_{a=0}^{n}\left(  -1\right)  ^{a}\binom{m+n+1}{n-a}%
b_{1}^{a}b_{2}^{-a-1}B_{n-a}\left(  y_{1}\right)  B_{m+a+1}\left(
y_{2}\right) \\
&  \quad-\sum\limits_{a=0}^{m}\left(  -1\right)  ^{a}\binom{m+n+1}{m-a}%
b_{2}^{a}b_{1}^{-a-1}B_{m-a}\left(  y_{2}\right)  B_{n+a+1}\left(
y_{1}\right)  .
\end{align*}
This may be written as%
\begin{align*}
T  &  =\sum\limits_{a=0}^{n}\left(  -1\right)  ^{n-a}\binom{m+n+1}{a}%
b_{1}^{n-a}b_{2}^{a-n-1}B_{a}\left(  y_{1}\right)  B_{m+n+1-a}\left(
y_{2}\right) \\
&  \quad-\sum\limits_{a=0}^{m}\left(  -1\right)  ^{m-a}\binom{m+n+1}{a}%
b_{2}^{m-a}b_{1}^{a-m-1}B_{a}\left(  y_{2}\right)  B_{m+n+1-a}\left(
y_{1}\right)  .
\end{align*}
Without loss of generality we may assume that $n\geq m$; in this case we
divide the first sum into two parts, from $0$ to $m$ and $m+1$ to $n$, and the
results is%
\begin{align*}
&  \sum\limits_{a=0}^{m}\left(  -1\right)  ^{n-a}\binom{m+n+1}{a}b_{1}%
^{n-a}b_{2}^{a-n-1}B_{a}\left(  y_{1}\right)  B_{m+n+1-a}\left(  y_{2}\right)
\\
&  \quad=\sum\limits_{a=n+1}^{m+n+1}\left(  -1\right)  ^{m+1-a}\binom
{m+n+1}{a}b_{1}^{a-m-1}b_{2}^{m-a}B_{m+n+1-a}\left(  y_{1}\right)
B_{a}\left(  y_{2}\right)
\end{align*}
and
\begin{align*}
&  \sum\limits_{a=m+1}^{n}\left(  -1\right)  ^{n-a}\binom{m+n+1}{a}b_{1}%
^{n-a}b_{2}^{a-n-1}B_{a}\left(  y_{1}\right)  B_{m+n+1-a}\left(  y_{2}\right)
\\
&  \quad=\sum\limits_{a=m+1}^{n}\left(  -1\right)  ^{m+1-a}\binom{m+n+1}%
{a}b_{1}^{a-m-1}b_{2}^{m-a}B_{m+n+1-a}\left(  y_{1}\right)  B_{a}\left(
y_{2}\right)
\end{align*}
(for $m=n$ the above sum vanishes). Thus,
\begin{equation}
T=\frac{1}{b_{1}^{m+1}b_{2}^{n+1}}\sum\limits_{a=0}^{m+n+1}\left(  -1\right)
^{m+1-a}\binom{m+n+1}{a}b_{1}^{a}b_{2}^{m+n+1-a}B_{m+n+1-a}\left(
y_{1}\right)  B_{a}\left(  y_{2}\right)  . \label{29}%
\end{equation}
So, combining (\ref{25}) and (\ref{29}), the desired result follows.
\end{proof}

Observe that starting from the left-hand side of (\ref{24}) and proceeding as
in the proof of (\ref{29}), the right-hand side of (\ref{24}) turns into%
\[
\frac{1}{b_{1}^{m+1}b_{2}^{n+1}}\sum\limits_{a=0}^{m+n+1}\left(  -1\right)
^{m+1-a}\binom{m+n+1}{a}b_{1}^{a}b_{2}^{m+n+1-a}B_{m+n+1-a}\left(
b_{1}x+y_{1}\right)  B_{a}\left(  b_{2}x+y_{2}\right)  ,
\]
which implies that
\begin{align*}
&  \sum\limits_{a=0}^{m+n+1}\left(  -1\right)  ^{a}\binom{m+n+1}{a}b_{1}%
^{a}b_{2}^{m+n+1-a}B_{m+n+1-a}\left(  b_{1}x+y_{1}\right)  B_{a}\left(
b_{2}x+y_{2}\right) \\
&  \quad=\sum\limits_{a=0}^{m+n+1}\left(  -1\right)  ^{a}\binom{m+n+1}{a}%
b_{1}^{a}b_{2}^{m+n+1-a}B_{m+n+1-a}\left(  y_{1}\right)  B_{a}\left(
y_{2}\right)
\end{align*}
holds for all $x.$ Note that obtaining the relation above is not so clear
without using the integrals for the case $b_{1}\not =\pm b_{2}.$

Now we set $b_{1}=b_{2}=1$ in (\ref{24}). Then, by the fact $B_{m}\left(
1-x\right)  =\left(  -1\right)  ^{m}B_{m}\left(  x\right)  ,$
\[
T=\left(  -1\right)  ^{n}\sum\limits_{a=0}^{m+n+1}\binom{m+n+1}{a}%
B_{m+n+1-a}\left(  1-y_{1}\right)  B_{a}\left(  y_{2}\right)  .
\]
We take $x=1-y_{1,}$ $y=y_{2}$ and $p=m+n+1$ in the well-known identity%
\begin{equation}
\sum\limits_{a=0}^{p}\binom{p}{a}B_{p-a}\left(  x\right)  B_{a}\left(
y\right)  =p\left(  x+y-1\right)  B_{p-1}\left(  x+y\right)  -\left(
p-1\right)  B_{p}\left(  x+y\right)  \label{23}%
\end{equation}
and we find that
\begin{align}
&  \sum\limits_{a=0}^{n}\left(  -1\right)  ^{a}\binom{m+n+1}{n-a}%
B_{n-a}\left(  x+y_{1}\right)  B_{m+a+1}\left(  x+y_{2}\right) \nonumber\\
&  \quad-\sum\limits_{a=0}^{m}\left(  -1\right)  ^{a}\binom{m+n+1}{m-a}%
B_{m-a}\left(  x+y_{2}\right)  B_{n+a+1}\left(  x+y_{1}\right) \nonumber\\
&  =\left(  -1\right)  ^{m}\left(  m+n+1\right)  \left(  y_{2}-y_{1}\right)
B_{m+n}\left(  y_{1}-y_{2}\right)  +\left(  -1\right)  ^{m}\left(  m+n\right)
B_{m+n+1}\left(  y_{1}-y_{2}\right)  . \label{28}%
\end{align}
Note that (\ref{28}) reduces to \cite[Propositon 2]{ad} for $y_{1}=y_{2}$. The
case $b_{1}=-b_{2}=-1$ in (\ref{24}) is equivalent to (\ref{28}) which can be
seen by setting $1-y_{1}$ instead of $y_{1}$ and using $B_{n}\left(
1-y\right)  =\left(  -1\right)  ^{n}B_{n}\left(  y\right)  $.

We mention also that the right-hand side of (\ref{32}) becomes simple for real
numbers $y_{l},$ $b_{l}=\left(  1-2y_{l}\right)  /q,$ $\left(  y_{l}%
\not =1/2\right)  ,$ $1\leq l\leq r,$ and $x=q\not =0.$ Then%
\begin{align*}
&  \widehat{I}_{n_{1},...,n_{r}}\left(  q;\frac{1-2y_{1}}{q},\ldots
,\frac{1-2y_{r}}{q};y_{1},...,y_{r}\right) \\
&  =\sum\limits_{a=0}^{n_{1}+\cdots+n_{r-1}}\left(  -1\right)  ^{a}%
\sum\limits_{j_{1}+\cdots+j_{r-1}=a}\binom{a}{j_{1},...,j_{r-1}}b_{1}^{j_{1}%
}\cdots b_{r-1}^{j_{r-1}}b_{r}^{-a-1}\\
&  \quad\times\frac{\left(  \left(  -1\right)  ^{n_{1}+\cdots+n_{r}%
+1}-1\right)  }{\left(  n_{1}-j_{1}\right)  !\cdots\left(  n_{r}+a+1\right)
!}B_{n_{1}-j_{1}}\left(  y_{1}\right)  \cdots B_{n_{r-1}-j_{r-1}}\left(
y_{r-1}\right)  B_{n_{r}+a+1}\left(  y_{r}\right)
\end{align*}
since $B_{n_{l}-j_{l}}\left(  b_{l}q-y_{l}\right)  =B_{n_{l}-j_{l}}\left(
1-y_{l}\right)  =\left(  -1\right)  ^{n_{l}-j_{l}}B_{n_{l}-j_{l}}\left(
y_{l}\right)  $ and $j_{1}+\cdots+j_{r-1}=a.$ Therefore if $n_{1}+\cdots
+n_{r}+1$ is even, then
\begin{equation}
\widehat{I}_{n_{1},...,n_{r}}\left(  q;\frac{1-2y_{1}}{q},\ldots
,\frac{1-2y_{r}}{q};y_{1},...,y_{r}\right)  =0, \label{17a}%
\end{equation}
and if $n_{1}+\cdots+n_{r}+1$ is odd, then
\begin{align}
&  \widehat{I}_{n_{1},...,n_{r}}\left(  q;\frac{1-2y_{1}}{q},\ldots
,\frac{1-2y_{r}}{q};y_{1},...,y_{r}\right) \nonumber\\
&  =-2q\sum\limits_{a=0}^{n_{1}+\cdots+n_{r-1}}\left(  -1\right)  ^{a}%
\frac{\left(  1-2y_{r}\right)  ^{-a-1}}{\left(  n_{r}+a+1\right)  !}%
B_{n_{r}+a+1}\left(  y_{r}\right) \nonumber\\
&  \quad\times\sum\limits_{j_{1}+\cdots+j_{r-1}=a}\binom{a}{j_{1},...,j_{r-1}%
}\prod\limits_{l=1}^{r-1}\frac{\left(  1-2y_{l}\right)  ^{j_{l}}}{\left(
n_{l}-j_{l}\right)  !}B_{n_{l}-j_{l}}\left(  y_{l}\right)  . \label{17b}%
\end{align}
For example, we have
\begin{align*}
&  \int\limits_{0}^{1}B_{3}\left(  -z+1\right)  B_{4}\left(  3z-1\right)
B_{16}\left(  5z-2\right)  dz=0,\\
&  \frac{1}{3!4!15!}\int\limits_{0}^{1}B_{3}\left(  -z+1\right)  B_{4}\left(
3z-1\right)  B_{15}\left(  -3z+2\right)  dz\\
&  \quad=-2\sum\limits_{a=0}^{7}\frac{B_{16+a}\left(  2\right)  }{\left(
16+a\right)  !}\sum\limits_{i=0}^{a}\binom{a}{i}3^{-i-1}\frac{B_{3-i}%
B_{4-a+i}\left(  -1\right)  }{\left(  3-i\right)  !\left(  4-a+i\right)  !}.
\end{align*}

\subsection{Generalized Bernoulli polynomials}

The counterpart of Proposition \ref{th-i2} for generalized Bernoulli
polynomials $B_{n_{l},\chi_{l}}\left(  b_{l}z+y_{l}\right)  $ can be also
expressed for non-principal primitive characters $\chi_{l}$ of modulus
$k_{l},$ $1\leq l\leq r.$ But in this case, $n_{1},\ldots,n_{r}$ must be
$\geq1$ since the degree of $B_{n,\chi}\left(  z\right)  $\ is not greater
than $n-1.$ Moreover, the summation from $a=0$ to $n_{1}+\cdots+n_{r-1}$\ on
the right-hand side of (\ref{32}) can be replaced by the summation from $a=0$
to $n_{1}+\cdots+n_{r-1}-\left(  r-1\right)  $. Under this circumstances the
analogues of (\ref{17a}) and (\ref{17b}) are valid according to $\chi
_{1}\left(  -1\right)  \cdots\chi_{r}\left(  -1\right)  \left(  -1\right)
^{n_{1}+\cdots+n_{r}+1}=1$ or $-1,$ but in this case $b_{l}=-2y_{l}/q,$
$\left(  y_{l}\not =0\right)  ,$ $1\leq l\leq r.$ Furthermore, first writing
the analogue of (\ref{25}) then adopting the arguments in the proof of
(\ref{29}), it can be seen for all $m,n\geq1$ that
\begin{align}
&  \sum\limits_{a=0}^{n}\left(  -1\right)  ^{a}\binom{m+n+1}{n-a}b_{1}%
^{a}b_{2}^{-a-1}B_{n-a,\chi_{1}}\left(  b_{1}x+y_{1}\right)  B_{m+a+1,\chi
_{2}}\left(  b_{2}x+y_{2}\right) \nonumber\\
&  \quad-\sum\limits_{a=0}^{m}\left(  -1\right)  ^{a}\binom{m+n+1}{m-a}%
b_{2}^{a}b_{1}^{-a-1}B_{m-a,\chi_{2}}\left(  b_{2}x+y_{2}\right)
B_{n+a+1,\chi_{1}}\left(  b_{1}x+y_{1}\right) \nonumber\\
&  \ =\frac{\left(  -1\right)  ^{m+1}}{b_{1}^{m+1}b_{2}^{n+1}}\sum
\limits_{a=0}^{m+n+1}\left(  -1\right)  ^{a}\binom{m+n+1}{a}b_{1}^{a}%
b_{2}^{m+n+1-a}B_{m+n+1-a,\chi_{1}}\left(  y_{1}\right)  B_{a,\chi_{2}}\left(
y_{2}\right)  . \label{36}%
\end{align}

We notice that if $\left(  -1\right)  ^{m+n}\chi_{1}\left(  -1\right)
\chi_{2}\left(  -1\right)  =1$ and $y_{1}=y_{2}=0,$\textbf{ }then the
right-hand side of (\ref{36}) vanishes, since $B_{m+n+1-a,\chi_{1}}%
B_{a,\chi_{2}}=0$\ by (\ref{3}). In the case $\left(  -1\right)  ^{m+n}%
\chi_{1}\left(  -1\right)  \chi_{2}\left(  -1\right)  =-1$ and $y_{1}=y_{2}%
=0$\textbf{ }the sum on the right-hand side of (\ref{36}) is closely related
to the reciprocity formulas of character Dedekind sums. For instance, from
Corollary \ref{rp3} and (\ref{36}), we have\textbf{ }%
\begin{align*}
&  \left(  p+1\right)  \left(  bc^{p}\text{ }\widehat{S}_{p}\left(
b,c:\overline{\chi_{1}},\chi_{2}\right)  +\frac{{}}{{}}cb^{p}\text{
}\widehat{S}_{p}\left(  c,b:\overline{\chi_{2}},\chi_{1}\right)  \right) \\
&  \quad=\sum\limits_{a=0}^{n-1}\binom{m+n+1}{n-a}c^{m+a+1}b^{n-a}%
B_{n-a,\chi_{1}}\left(  -cx\right)  B_{m+a+1,\chi_{2}}\left(  bx\right) \\
&  \qquad+\sum\limits_{a=0}^{m-1}\binom{m+n+1}{m-a}b^{n+a+1}c^{m-a}%
B_{m-a,\chi_{2}}\left(  bx\right)  B_{n+a+1,\chi_{1}}\left(  -cx\right) \\
&  \quad=\sum\limits_{j=0}^{p+1}\binom{p+1}{j}c^{j}b^{p+1-j}B_{p+1-j,\chi_{1}%
}B_{j,\chi_{2}}%
\end{align*}
for $\left(  -1\right)  ^{p+1}\chi_{1}(-1)\chi_{2}(-1)=1$ and $p=m+n$ $\left(
m,n\geq1\right)  $.

\subsection{Laplace transform of Bernoulli function}

Let $Re\left(  s\right)  >0$ and $|s/t|<2\pi.$ Setting $f(u)=e^{-su}$ and
$P_{n}\left(  u\right)  =\overline{B}_{n}\left(  tu+y\right)  ,$ $n\geq1,$ in
(\ref{th-i}) gives%
\begin{align}
\frac{1}{n!}\int\limits_{0}^{x}e^{-su}\overline{B}_{n}\left(  tu+y\right)  du
&  =\sum\limits_{a=0}^{\mu}\frac{s^{a}t^{-a-1}}{\left(  n+a+1\right)
!}\left\{  e^{-sx}\overline{B}_{n+a+1}(tx+y)-\overline{B}_{n+a+1}(y)\right\}
\nonumber\\
&  +\left(  \frac{s}{t}\right)  ^{\mu+1}\frac{1}{\left(  n+\mu+1\right)
!}\int\limits_{0}^{x}e^{-su}\overline{B}_{n+\mu+1}\left(  tu+y\right)  du.
\label{37}%
\end{align}
Since the function $\overline{B}_{m}\left(  u\right)  =B_{m}\left(  u-\left[
u\right]  \right)  $ is bounded, the integrals in (\ref{37}) converge
absolutely and $e^{-sx}\overline{B}_{n+a+1}(tx+y)$ tends to $0$ as
$x\rightarrow\infty.$ Then, letting $x\rightarrow\infty,$ we have%
\begin{align}
\frac{1}{n!}\int\limits_{0}^{\infty}e^{-su}\overline{B}_{n}\left(
tu+y\right)  du  &  =-\frac{t^{n}}{s^{n+1}}\sum\limits_{a=0}^{\mu}%
\frac{\overline{B}_{n+a+1}(y)}{\left(  n+a+1\right)  !}\frac{s^{n+a+1}%
}{t^{n+a+1}}\nonumber\\
&  +\left(  \frac{s}{t}\right)  ^{\mu+1}\frac{1}{\left(  n+\mu+1\right)
!}\int\limits_{0}^{\infty}e^{-su}\overline{B}_{n+\mu+1}\left(  tu+y\right)
du. \label{39}%
\end{align}
From (\ref{0}) the sum in (\ref{39}) converges absolutely for $|s/t|<2\pi$ as
$\mu\rightarrow\infty.$ Also the sequence of the functions $g_{\mu}\left(
u\right)  =s^{\mu}\overline{B}_{\mu}\left(  tu+y\right)  /\mu!t^{\mu}$
converges uniformly to $0$ for $|s/t|<2\pi$. Thus, letting $\mu\rightarrow
\infty$ we find that
\begin{equation}
\frac{1}{n!}\int\limits_{0}^{\infty}e^{-su}\overline{B}_{n}\left(
tu+y\right)  du=-\frac{t^{n}}{s^{n+1}}\sum\limits_{a=n+1}^{\infty}%
\frac{\overline{B}_{a}(y)}{a!}\frac{s^{a}}{t^{a}}. \label{16a}%
\end{equation}
Since $\overline{B}_{a}(y)=B_{a}(\left\{  y\right\}  )$ for $a\geq2$ , where
$\left\{  y\right\}  =y-\left[  y\right]  ,$ we have
\[
\sum\limits_{a=n+1}^{\infty}\frac{\overline{B}_{a}(y)}{a!}u^{a}=-\sum
\limits_{a=0}^{n}\frac{B_{a}(\left\{  y\right\}  )}{a!}u^{a}+\frac
{ue^{\left\{  y\right\}  u}}{e^{u}-1}%
\]
by (\ref{0}). Therefore, we arrive at the Laplace transform of $\overline
{B}_{n}\left(  tu+y\right)  $%
\begin{equation}
\int\limits_{0}^{\infty}e^{-su}\overline{B}_{n}\left(  tu+y\right)
du=n!\frac{t^{n}}{s^{n+1}}\left(  \sum\limits_{a=0}^{n}\frac{B_{a}(\left\{
y\right\}  )}{a!}\frac{s^{a}}{t^{a}}-\frac{s}{t}\frac{e^{\left\{  y\right\}
s/t}}{e^{s/t}-1}\right)  \label{16}%
\end{equation}
for all $Re\left(  s\right)  >0$ and $n\geq1,$ by analytic continuation.
Differentiating $r$ times both sides of (\ref{16}) with respect to $s,$ then
using the well-known identity $\sum\nolimits_{r=0}^{m}\binom{m}{r}B_{m-r}%
x^{r}=B_{m}\left(  x\right)  $ we deduce that
\begin{align*}
&  \int\limits_{0}^{\infty}e^{-su}B_{m}\left(  u\right)  \overline{B}%
_{n}\left(  u\right)  du\\
&  \quad=\sum\limits_{r=0}^{m}\binom{m}{r}B_{m-r}\left(  \sum\limits_{a=0}%
^{n}\binom{n}{a}\frac{\left(  n+r-a\right)  !}{s^{n+1+r-a}}B_{a}-n!\left(
-1\right)  ^{r}\frac{d^{r}}{ds^{r}}\frac{s^{-n}}{e^{s}-1}\right)  .
\end{align*}

By the similar way, we have
\[
\frac{1}{n!}\int\limits_{0}^{\infty}e^{-su}\overline{B}_{n,\chi}\left(
tu\right)  du=\frac{1}{s}\sum\limits_{a=0}^{n}\frac{B_{a,\chi}}{a!}\left(
\frac{t}{s}\right)  ^{n-a}-\frac{t^{n-1}}{s^{n}}\sum\limits_{j=0}^{k-1}%
\frac{\overline{\chi}\left(  j\right)  e^{js/t}}{e^{ks/t}-1}%
\]
for $Re\left(  z\right)  >0,$ $n\geq1,$ and a non-principal primitive
character $\chi$ of modulus $k.$


\begin{thebibliography}{99}                                                                                               %


\bibitem {ad}T. Agoh and K. Dilcher, Integrals of products of Bernoulli
polynomials, J. Math. Anal. Appl. 381 (2011), 10--16.

\bibitem {ap}T. M. Apostol, Generalized Dedekind sums and transformation
formulae of certain Lambert series, Duke Math. J. 17 (1950) 147--157.

\bibitem {b2}B. C. Berndt, Character transformation formulae similar to those
for the Dedekind Eta--function, in `Analytic Number Theory', Proc. Sym. Pure
Math. XXIV, Amer. Math. Soc., Providence, R. I.,(1973) 9--30.

\bibitem {b5}B. C. Berndt, Character analogues of Poisson and Euler--MacLaurin
summation formulas with applications, J. Number Theory 7 (1975) 413--445.

\bibitem {b6}B. C. Berndt, Reciprocity theorems for Dedekind sums and
generalizations, Advances in Mathematics 23 (1977) 285--316.

\bibitem {ck}M. Can and V. Kurt, Character analogues of certain Hardy--Berndt
sums, Int. J. Number Theory 10 (2014) 737--762.

\bibitem {cd}M. Can and M. C. Da\u{g}l\i, On reciprocity formula of character
Dedekind sums: in International Conference on Algebra and Number Theory, 05-08
August 2014, Samsun, Turkey.

\bibitem {c}L. Carlitz, Note on the integral of the product of several
Bernoulli polynomials, J. London Math. Soc. 34 (1959) 361--363.

\bibitem {cck}M. Cenkci, M. Can, V. Kurt, Degenerate and character Dedekind
sums, J. Number Theory 124 (2007) 346--363.

\bibitem {dc}M. C. Da\u{g}l\i\ and M. Can, On reciprocity formulas for
Apostol's Dedekind sums and their analogues, J. Integer Seq. 17 (2014) Article 14.5.4.

\bibitem {em}O. Espinosa and V. H. Moll, The evaluation of Tornheim double
sums, Part 1, J. Number Theory 116 (2006) 200--229.

\bibitem {kg}K. Girstmair, On the asymptotic behavior of Dedekind sums, J.
Integer Seq. 17 (2014) Article 14.7.7.

\bibitem {h}Y. Hamahata, Dedekind sums with a parameter in finite fields,
Finite Fields Appl. 28 (2014) 57--66.

\bibitem {hkk}S. Hu, D. Kim, M.-S. Kim, On the integral of the product of four
and more Bernoulli polynomials, Ramanujan J. 33 (2014) 281--293.

\bibitem {ks}M.-S. Kim and J.-W. Son, On generalized Dedekind sums involving
quasi-periodic Euler functions, J. Number Theory 144 (2014) 267--280.

\bibitem {lz}J. Li and W. Zhang, A hybrid mean value involving the Dedekind
sums and exponential sums, Adv. Math. (China) 43 (2014), 4, 527--533.

\bibitem {lpz}J. Liu, H. Pan, Y. Zhang, On the integral of the product of the
Appell polynomials, Integral Transforms Spec. Funct.25 no. 9 (2014) 680--685.

\bibitem {mi}M. Mikol\'{a}s, Integral formulae of arithmetical characteristics
relating to the zeta-function of Hurwitz, Publ. Math. Debrecen 5 (1957) 44--53.

\bibitem {ml}J. B. Miller, The standard summation operator, the
Euler--MacLaurin sum formula, and the Laplace transformation, J. Australian
Math. Soc. (Series A) 39 (1985) 367--390.

\bibitem {mo}L. J. Mordell, Integral formulae of arithmetical character, J.
London Math. Soc. 33 (1958) 371--375.

\bibitem {n}N. E. N\"{o}rlund, Vorlesungen \"{u}ber Differenzenrechnung,
Springer-Verlag, Berlin, 1924.

\bibitem {rg}H. Rademacher, E. Grosswald, Dedekind sums, Carus Math.
Monographs 16, Math. Assoc. Amer. 1972.

\bibitem {w}J. C. Wilson, On Franel--Kluyver integrals of order three, Acta
Arith. 66 (1994) 71--87.

\bibitem {zh}W. Zhang and D. Han, A hybrid mean value of Dedekind sums and
Kloosterman sums, J. Number Theory 147 (2015), 861--870.
\end{thebibliography}
\end{document}